\documentclass[12pt]{article}
\usepackage{amsmath}
\usepackage{amsthm}
\usepackage{amssymb}
\usepackage{indentfirst}
\usepackage{enumerate}
\usepackage{graphicx}
\def\graphit#1{ \centerline{\includegraphics{#1}} }

\usepackage[hidelinks]{hyperref}
\newtheorem{theorem}{Theorem}
\newtheorem{corollary}[theorem]{Corollary}
\newtheorem{lemma}[theorem]{Lemma}
\newtheorem{proposition}[theorem]{Proposition}

\def\half{{1 \over 2}}

\def\Var{{\rm Var}}
\def\P{{\bf P}}
\def\E{{\bf E}}
\def\IZ{{\bf Z}}

\def\Geom{{\rm Geom}}
\def\un{\underbar}
\def\acc{{\rm acc}}
\def\eff{{\rm eff}}
\def\A{{\cal A}}
\def\B{{\cal B}}
\def\C{{\cal C}}

\def\ld{\ldots}

\begin{document}

\baselineskip=18pt

\begin{center}
\Large\bf
Quantifying the Speed-Up from
Non-Reversibility

in MCMC Tempering Algorithms
\end{center}

\centerline{Gareth O.\ Roberts \quad and \quad Jeffrey S.\ Rosenthal}

\centerline{University of Warwick \quad \quad \quad University of Toronto}

\bigskip
\centerline{(January, 2025)}

\bigskip

\section{Introduction}

Markov chain Monte Carlo (MCMC) algorithms are extremely important for
sampling from complicated high-dimensional densities, particularly in Bayesian Statistics
(see e.g.~\cite{handbook} and the many references therein).
Traditional MCMC algorithms like the Metropolis-Hastings algorithm
\cite{metropolis,hastings} are reversible.
However, in recent years there has been a trend towards using
versions
which introduce ``momentum'' and hence are non-reversible in some sense,
in an effort to avoid diffusive behaviour
\cite{suppress,radfordbetter,DHN,bouncy,zigzag}.

Many of the most challenging problems in sampling complex distributions come from multi-modality. In this context, the most successful approaches have been
simulated and parallel tempering algorithms. These algorithms add auxiliary temperature variables to improve mixing between modes
\cite{simtemp,swendsen,marinari,mcmcmc}.
Parallel tempering (which proceeds with a particle at each of a collection of auxiliary temperatures)  can be implemented by alternating even and odd index temperature
swap proposals.
The resulting algorithm is non-reversible (despite being constructed from reversible components)
and can create an effect of momentum for each particle as it moves around the temperature space, thus increasing efficiency
\cite{okabe,syed,ABC2022,ABC2023,ABC2024}.  This idea was also combined with efficient parallel implementation to create general-purpose software \cite{pigeons}.

In this paper, we provide a theoretical investigation of the extent to
which such non-reversibility can improve the efficiency of tempering MCMC algorithms.
In a certain diffusion-limit context, under strong assumptions,
we prove
that an optimally-scaled non-reversible MCMC sampler is indeed
more efficient than the corresponding optimally-scaled
reversible version, but the speed-up is only a modest 42\%.
This suggests that non-reversible MCMC is indeed worthwhile, but cannot
hope to convert intractable algorithms into tractable ones.


To demonstrate this, we first study (Section~\ref{sec-model}) a
simple Markov chain that can help model the reversible and non-reversible
behaviour of tempering algorithms.
We prove (Theorem~\ref{brownianthm}) that even a non-reversible-style
version of this chain still exhibits diffusive behaviour at
appropriate scaling.
We then consider (Section~\ref{sec-rescale}) rescaling space by
a factor of $\ell$, and describe certain ``optimal'' $\ell$ values.
We then apply (Section~\ref{sec-temper}) this reasoning to
tempering MCMC algorithms,
and prove
under the theoretical framework of \cite{atchade,sectemper}
that the reversible and non-reversible versions have different
efficiency curves (Theorem~\ref{curvethm})
and optimal scaling values (Theorem~\ref{optthm}),
leading to the 42\% improvement under optimality
(Corollary~\ref{corfortytwo}).

\section{A Double-Birth-Death Markov Chain}
\label{sec-model}

To study the effects of momentum on tempering, we first digress to study a
simple double-birth-death Markov chain, which may be of independent interest.

Consider the following discrete-time countable-state-space Markov chain,
which can be viewed as an infinite-size generalisation of the
simple finite example studied in \cite{DHN,mirageyer}.

This Markov chain has state space equal to $\IZ \times \{+,-\}$,
and transition probabilities given by
$\P((i,+),(i+1,+)) = A$,
$\P((i,+),(i-1,+)) = B$,
$\P((i,+),(i,-)) = C$,
$\P((i,-),(i-1,-)) = A$,
$\P((i,-),(i+1,-)) = B$,
$\P((i,-),(i,+)) = C$,
where $A+B+C=1$ are non-negative constants
with $C>0$.
(See Figure~1.)

\begin{center}
\fbox{ \includegraphics[height=2.5in]{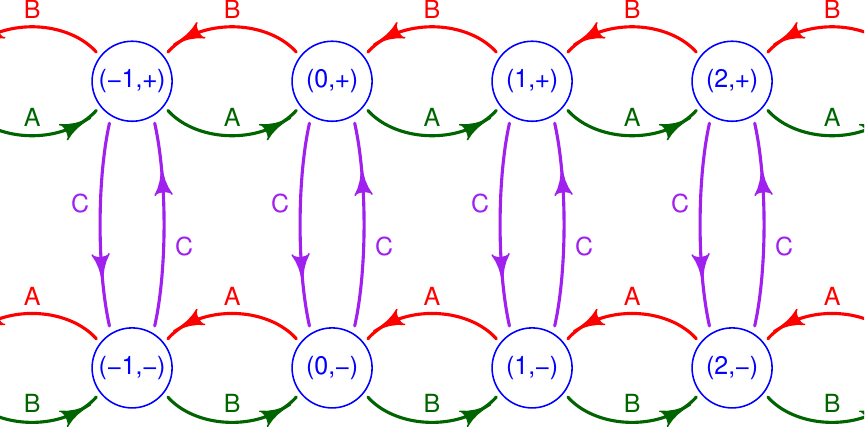} }
\par
{\bf Figure~1: Diagram of the double-birth-death Markov chain.}
\end{center}

This chain can be viewed as a ``lifting'' of a symmetric walk on $\IZ$.
That is, if states $(i,+)$ and $(i,-)$ are combined into a single state
$i$ for each $i\in\IZ$, with the chain equally likely to be at
$(i,+)$ or $(i,-)$, then this combined process is itself a
Markov chain which has the symmetric transition probabilities
$P(i,i+1) = P(i,i-1) = (A+B)/2$ and $P(i,i) = C$.
However, the full non-combined
chain has non-symmetric transitions whenever $A \not= B$.

Of course, if $A=B$, then this chain becomes a symmetric walk on both
$\IZ \times \{+\}$ and $\IZ \times \{-\}$.
By contrast,
if $A > B$, then it has a positive bias on $\IZ \times \{+\}$,
and a negative bias on $\IZ \times \{-\}$.
Indeed, if $B=0$, then it moves only positively on $\IZ \times \{+\}$,
and only negatively on $\IZ \times \{-\}$.

Despite the non-reversible-seeming
nature of this simple Markov chain,
the following result
(proved in Section~\ref{sec-proof})
gives a diffusion limit of a rescaled version, with
a full Functional Central Limit Theorem
(i.e., Donsker's Invariance Principle).
For notation, let the state of this chain at time $n$ be given by
$(X_n,Y_n)$ where $X_n$ is the horizontal integer and $Y_n$ is
the vertical $\pm$ value.
Then we have:

\begin{theorem}
\label{brownianthm}
Let $\{(X_n,Y_n)\}_{n=0}^\infty$ follow the Markov chain of Figure~1.
Let $Z_{M,\cdot}$ be the random process defined by
$Z_{M,t} := {1 \over \sqrt{M}} \, X_{\lfloor Mt \rfloor}$ for $t \ge 0$.
Then as $M\to\infty$,
the process $Z_{M,\cdot}$ converges
weakly to Brownian motion with zero drift and with
volatility parameter given by
\begin{equation}\label{vformula}
v = [(A-B)^2/C] + (A+B)
\, .
\end{equation}
In particular, for each fixed $t>0$,
as $M\to\infty$, the random variable
$Z_{M,t} := {1 \over \sqrt{M}} \, X_{Mt}$
converges weakly to the $N(0,vt)$ distribution.
\end{theorem}

In the special case where $A=B$,
this volatility becomes
\begin{equation}\label{volrev}
v
\ = \ [(A-B)^2/C] + (A+B)
\ = \ [0] + (2A)
\ = \ 2A
\, .
\end{equation}
In particular, if $C \searrow 0$ while $A=B \nearrow \half$,
then $v \to 1$ as for standard Brownian motion,
exactly as it should.
Or, in the special case where $B=0$,
this volatility becomes
\begin{equation}\label{volnonrev}
v
\ = \ [(A-B)^2/C] + (A+B)
\ = \ [A^2/(1-A)] + A
\ = \ A/(1-A)
\, .
\end{equation}

%
%

We note that
the intention of such a chain, at least when $A \gg B$, is to provide a
``momentum'' whereby the chain moves to the right along the top row for
long periods of time, and then to the left along the bottom row for long
periods of time, thus sweeping and exploring large regions more
efficiently than in the diffusive symmetric $A=B$ case.
That is indeed the case, over modest time intervals.
However, the chief observation of Theorem~\ref{brownianthm} is that
over larger time intervals, the chain will reverse direction sufficiently
that it will still exhibit diffusive behaviour, just on a larger time
scale.  By rescaling the chain appropriately, the diffusive behaviour can
still be identified and quantified, and hence directly
compared to the symmetric case, as we discuss below.

\section{Rescaling Space and Transition Probabilities}
\label{sec-rescale}

The study of MCMC algorithms includes scaling questions,
regarding how large their step sizes should be \cite{RGG, atchade}.
For the Markov chain in Figure~1, this corresponds to
expanding space by a constant factor of $\ell>0$, i.e.\ regarding the
adjacent points as being a distance $\ell$
apart rather than having unit distance.

In this context, the transition probabilities
$A=\A(\ell)$ and $B=\B(\ell)$ and $C=\C(\ell)$
also become functions of $\ell$
(still summing to~1 for each $\ell$).
The value of $v$ in
Theorem~\ref{brownianthm}
then becomes a corresponding function of $\ell$, too, i.e.\
$$
v \ = \ v(\ell)
\ = \
[(\A(\ell)-\B(\ell))^2/\C(\ell)] + (\A(\ell)+\B(\ell))
\, .
$$
Since distance is itself scaled by a factor of $\ell$, it follows
that the effective volatility of the rescaled process is now proportional
to $\ell^2 \, v(\ell)$, with $v(\ell)$ as above.

Now, it is known \cite[Theorem~1]{sectemper}
that limiting diffusions are most efficient in
terms of minimising Monte Carlo variance when their volatility is largest.
Hence, to make a MCMC algorithm most efficient, we need to maximise
that effective volatility, henceforth referrred to as the
efficiency function $\eff(\ell) := \ell^2 \, v(\ell)$.

Of course, this maximisation depends on the functional
dependence of $\A(\ell)$ and $\B(\ell)$,
i.e.\ how the transition probabilities are affected by the spacing $\ell$.
However, there will often exist an optimal value $\ell_* > 0$ which maximises $\eff(\ell)$.
For example, we have the following results.

\begin{proposition}
If $\A(\ell)$ is positive, log-concave, $C^1$, and non-increasing, then there exists a unique  $\ell_* > 0$ which maximises $\eff(\ell)$.
    \end{proposition}
\begin{proof}
    Let $f(\ell) = \log \A(\ell)$, so that $\eff(\ell) = \ell^2 \A(\ell) = \ell^2 e^{f(\ell)}$.  Then any stationary point of $\eff(\ell)$ must have $\eff'(\ell)=0$, so that
    $$
  2\ell e^{f(\ell )} + \ell^2 f'(\ell ) e^{f(\ell )} \ = \ 0 \, .
    $$
    Eliminating the point $\ell = 0$ which is clearly a minimum, we need to satisfy
    $$  -f'(\ell) = 2/\ell \ .
    $$
    Now, the right-hand side of this equation is strictly decreasing from $\infty $ to $0$, and the left-hand side is non-decreasing from a finite value.  Hence, since both functions are continuous, there must exist a unique stationary point $\ell_* > 0$. Since $\eff(\ell )$ is non-negative and $\eff(0)=0$, it follows that  $\ell_*$ is a maximum as required. 
\end{proof}
For example, suppose that $\A(\ell)=2 \Phi (-c\ell /2)$ for the cumulative normal distribution function $\Phi$, it is easy to check that 
$\A(\ell)$ is log-concave.  We will see the relevance of this case in the next section.

\begin{proposition}
If $\B(\ell) \equiv 0$,
and $\A(\ell)$ is continuous,
and $\lim_{\ell \searrow 0} {\log[1-\A(\ell)] \over \log(\ell)} < 2$,
and $\lim_{\ell\to\infty} \ell^2 \A(\ell) = 0$,
then there exists $\ell_* \in (0,\infty)$ such that
$\eff(\ell)$ is maximised at $\ell=\ell_*$.
\end{proposition}

\begin{proof}
The assumptions imply that to first order
as $\ell \searrow 0$,
$1-\A(\ell) = \ell^{\eta}$ for some
$\eta<2$,
i.e.\ $\A(\ell) = 1 - \ell^{\eta}$.
Then
$\eff(\ell)
= \ell^2 \, \A(\ell)/(1-\A(\ell))
= \ell^2 \, (1 - \ell^{\eta}) / \ell^{\eta}
= \ell^{2-\eta} - \ell^2$.
Since $\eta<2$, this implies that
$\eff(\ell)>0$ for all small positive $\ell$.
However,
$\lim_{\ell\to\infty} \eff(\ell)
= \lim_{\ell\to\infty} \ell^2 \, \A(\ell)/(1-\A(\ell))
\le \lim_{\ell\to\infty} \ell^2 \, \A(\ell)
= 0$ by assumption.
Hence, by continuity and the Extreme Value Theorem,
$\eff(\ell)$ must take its maximum at some $\ell_*\in(0,\infty)$.
\end{proof}

However, the real value of these rescaling operations
is to optimise MCMC algorithms like tempering, as we now discuss.

\section{Application to Tempering Algorithms}
\label{sec-temper}

Tempering algorithms,
including simulated tempering~\cite{marinari}
and parallel tempering~\cite{mcmcmc},
 are now widely used to improve MCMC by allowing
mixing between modes.  They involve specifying a sequence of
temperature values which increase from one (corresponding to the original
``coldest'' distribution) to some fixed large value (corresponding to a
``hottest'' distribution which facilitates easy mixing between modes).

Typically, we define inverse temperatures
$0 \leq \beta_N<\beta_{N-1}<\ldots <\beta_1<\beta_0=1$,
and let $\pi_\beta(x)\propto \left[\pi(x)\right]^\beta$
be a power of the target density $\pi(x)$.
Simulated Tempering (ST) augments the original state space with a
one-dimensional component indicating the current inverse temperature level,
thus creating a $(d+1)$-dimensional chain with target
$\pi(\beta,x)\propto K(\beta)\pi(x)^{\beta}$,
where ideally $K(\beta)=\left[\int_{x}\pi(x)^{\beta}\mbox{d}x\right]^{-1}$
so that $\beta$ has uniform marginal.
By contrast, Parallel Tempering (PT)
runs a chain on $N$ copies of the state space,
each at a different temperature, with target
$\pi_N(x_0,x_1,\ld,x_N) \propto
\pi_{\beta_0}(x_0) \, \pi_{\beta_1}(x_1) \ld \pi_{\beta_N}(x_N)$.
Each algorithm attempts to use the hotter temperatures to help the
chain move between modes, and thus better sample the original cold
temperature target $\pi(x) = \pi_0(x)$.

For tempering algorithms to be useful, they have to move fairly
efficiently between the extreme temperatures.  In particular, the rate of
temperature round-trips from coldest to hottest to coldest is often
 a good measure of a tempering algorithm's efficiency \cite{syed}.
To study this, we use the theoretical framework developed in
\cite{atchade,sectemper}.
This framework considers tempering within a single mode of the
target distribution, such
that spatial mixing is very easy and can be considered to happen immediately.
(This corresponds to the ``Efficient Local Exploration'' (ELE)
assumption in~\cite{syed}.)
Furthermore, it assumes the same product i.i.d.\ structure as
for theoretical efficiency study of random-walk Metropolis
algorithms as in~\cite{RGG,statsci}.


It is known
for tempering algorithms~\cite{atchade,sectemper},
like for random-walk Metropolis (RWM) algorithms~\cite{RGG,statsci},
that under these strong assumptions,
in the limit as the dimension $d\to\infty$,
proposed moves at scaling $\ell$
are accepted with asymptotic probability
$2 \, \Phi(-c \, \ell / 2)$ for a
specific problem-dependent constant $c>0$
(here $c=\sqrt{I}$ in the notation of \cite{atchade,sectemper}),
where
$\Phi(z) = \int_{-\infty}^z {1 \over \sqrt{2\pi}} \, e^{-u^2/2} \, du$
is the cumulative distribution fuction (cdf)
of the standard normal distribution,
with inverse function $\Phi^{-1}$.

In the usual reversible version of tempering,
the proposed moves would be to increase or decrease the temperature index
by $1$, with probability 1/2 each.
In our context, this implies that
$\A(\ell)=\B(\ell) = (1/2) [2 \, \Phi(-c \, \ell / 2)]
= \Phi(-c \, \ell / 2)$.
Hence, the overall proposal acceptance rate
then becomes
$\acc(\ell) \equiv \A(\ell)+\B(\ell) = 2 \, \Phi(-c \, \ell / 2)$,
as discussed in \cite{atchade,sectemper}.

By contrast, the non-reversible momentum version of tempering
would always propose to increase the temperature index by~1
on $\IZ \times \{+\}$,
or decrease by~1 on $\IZ \times \{-\}$).
This corresponds to
$\B(\ell) \equiv 0$
and $\acc(\ell) = \A(\ell)
= (1) \, 2 \, \Phi(-c \, \ell / 2)
= 2 \, \Phi(-c \, \ell / 2)$.

We now derive various results about the relationship between
efficiency and acceptance rate for the reversible and non-reversible
versions of these tempering algorithms under these assumptions,
as illustrated in Figure~2 (for the case $c=1$).
Note that we measure relative efficiency here in terms of the volatility
of the limiting diffusion as justified by
\cite[Theorem~1]{sectemper} as discussed above.
However, we will see in Section~\ref{sec-sim} below that
simulated temperature round-trip rates do indeed follow
these relative efficiency curves very closely.

\begin{theorem}\label{curvethm}
Consider a tempering algorithm under the assumptions of
\cite{atchade,sectemper} as above.
Then in the limit as the dimension $d\to\infty$,
the efficiency measure $\eff(\ell)$ is related
to the acceptance rate $\acc(\ell)$ as follows:
\hfil\break (i) In the reversible case,
$\eff(\ell) = \acc(\ell) {4 \over c^2}
		\left[ \Phi^{-1}(\acc(\ell)/2)) \right]^2$.
\hfil\break (ii) In the non-reversible case,
$\eff(\ell) = {\acc(\ell) \over 1-\acc(\ell)} {4 \over c^2}
		\left[ \Phi^{-1}(\acc(\ell)/2)) \right]^2$.
\end{theorem}

\begin{proof}
In the reversible case, by~\eqref{volrev},
we have
$v(\ell) = 2 \, A(\ell) = \acc(\ell)$, whence
$\eff(\ell) = \ell^2 \, v(\ell) = \ell^2 \, \acc(\ell)$.

In the non-reversible case, by~\eqref{volnonrev},
we have
$v(\ell) = {A(\ell) \over 1-A(\ell)} = {\acc(\ell) \over 1-\acc(\ell)}$,
whence $\eff(\ell) = \ell^2 \, v(\ell) = \ell^2 \,
{\acc(\ell) \over 1-\acc(\ell)}$.

In either case, we have $\acc(\ell) = 2 \, \Phi(-c \, \ell / 2)$.
Inverting this,
$\ell = -{2 \over c} \, \Phi^{-1}(\acc(\ell)/2)$.
Plugging this formula into the expressions for $\eff(\ell)$,
the two formulae follow.
\end{proof}

\bigskip
\bigskip
\graphit{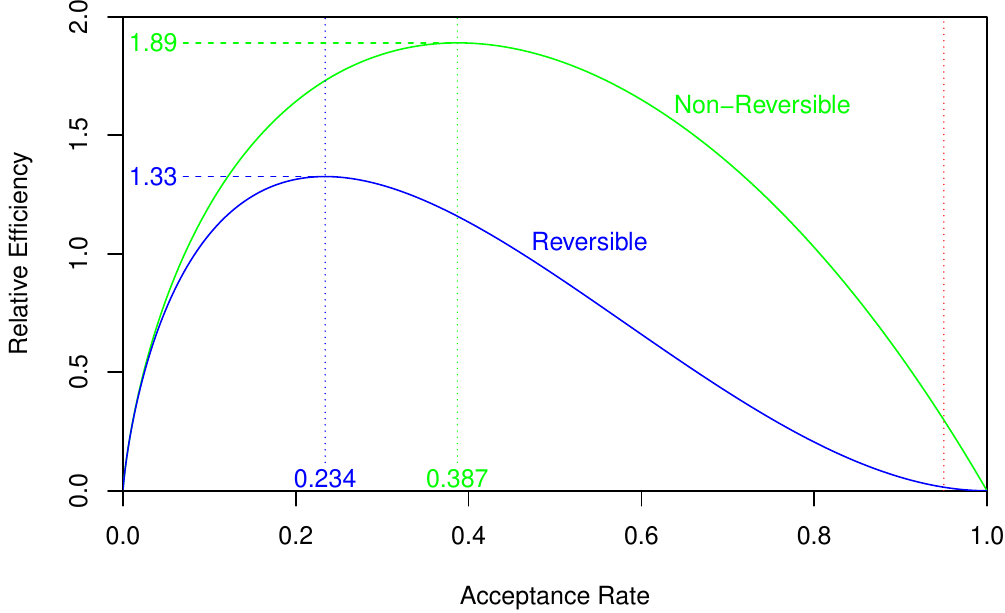}
\vskip -0.8cm
$\big.$
\begin{quote}
\baselineskip=12pt
{\bf Figure~2:}\
Efficiency curves for non-reversible (top, green)
and reversible (bottom, blue) algorithms
from Theorem~\ref{curvethm}
when $c{=}1$,
with their
optimal scaling values (dashed lines)
from Theorem~\ref{optthm},
including the 1.89/1.33 ratio
from Corollary~\ref{corfortytwo},
and the infinity-tending ratio towards the right (red)
from Corollary~\ref{infcor}.
\end{quote}
\bigskip

We can then maximise the efficiency curves from
Theorem~\ref{curvethm}, as follows:

\begin{theorem}\label{optthm}
Consider the tempering algorithm described above, under the strong
assumptions from \cite{atchade,sectemper}.
Then:
\hfil\break {\bf (i)}
\cite{atchade,sectemper} \
In the reversible case where
$\A(\ell)=\B(\ell) = \Phi(-c \, \ell / 2)$,
the efficiency fuction $\eff(\ell)$ is maximised
(to three decimal points)
by choosing $\ell = \ell_* \doteq 2.38 / c$,
whence $\A(\ell_*) \doteq 0.117$
and $\acc(\ell_*) = 2 \, \A(\ell_*) \doteq 0.234$
and $\eff(\ell_*) \doteq 1.33 / c^2$.
By contrast:
\hfil\break {\bf (ii)} In the non-reversible case where $\B(\ell) \equiv 0$
and $\A(\ell) = 2 \, \Phi(-c \, \ell / 2)$,
the efficiency function $\eff(\ell)$ is maximised
by choosing $\ell = \ell_{**} \doteq 1.73 / c$,
whence $\acc(\ell_{**}) \doteq 0.387$
and $\eff(\ell_{**}) \doteq 1.89 / c^2$.
\end{theorem}

\begin{proof}
In the reversible case,
we need to choose $\ell$ to maximise
$$
\eff(\ell)
\ = \ \ell^2 v(\ell)
\ = \ \ell^2 \, 2 \, \A(\ell)
\ = \ 2 \, \ell^2 \, \Phi(-c \, \ell / 2)
\, .
$$
Letting $s=c\ell/2$, this is equivalent to maximising $s^2 \, \Phi(-s)$.
Numerically, the latter is maxmised at $s = s_* \doteq 1.1906$,
corresponding to $\ell = \ell_* = 2 s_* / c \doteq 2.3812 / c$, whence
$$
\eff(\ell_*)
\ = \ 2 \ell_*^2 \, \Phi( - s_* )
\ \doteq \ 2(2.3812)^2 c^{-2} \, \Phi( - 1.1906 )
\ \doteq \ 1.3257 / c^2
\, .
$$
Then $\A(\ell_*) = \Phi(-s_*) \doteq
\Phi(-1.1906) \doteq 0.117$.
The corresponding optimal
acceptance rate is thus $\A(\ell_*)+\B(\ell_*) = 2 \, \A(\ell_*) \doteq 0.234$,
just as for random-walk Metropolis algorithms~\cite{RGG,statsci}.


By contrast, in the non-reversible case,
we need to choose $\ell$ to maximise
$$
\eff(\ell) = \ell^2 v(\ell)
= \ell^2 \, \A(\ell) / (1-\A(\ell))
= \ell^2 \, 2 \, \Phi(-c \, \ell / 2) / [1 - 2 \, \Phi(-c \, \ell / 2)]
\, .
$$
Letting $s=c\ell/2$, this is equivalent to maximising
$s^2 \, \Phi(-s) / [1-2\Phi(-s)]$.
Numerically, the latter is maxmised at $s = s_{**} \doteq 0.8643$,
corresponding to $\ell = \ell_{**} = 2 s_{**}/c \doteq 1.7285 / c$, whence
$$
\eff(\ell_{**})
\ = \ \ell_{**}^2 \, v(\ell_{**})
\ = \ \ell_{**}^2 \, 2 \, \Phi(-c \, \ell_{**} / 2)
			/ [1 - 2 \, \Phi(-c \, \ell_{**} / 2)]
$$
$$
\ \doteq \ (1.7285)^2 \, 2 \, c^{-2} \, \Phi( - 0.8643 ) / [1-\Phi(-0.8643)]
\ \doteq \ 1.8896 / c^2
\, .
$$
Then
$\acc(\ell_{**}) =
2 \, \A(\ell_{**}) \doteq
2 \, \Phi( - c \ell_{**} / 2 )
= 2 \, \Phi( - c (1.7285) / 2c ) = 2 \, \Phi(-0.8642) \doteq 0.387$,
as claimed.
\end{proof}


%
%

Theorem~\ref{optthm} provides some practical guidance when running
tempering algorithms.
In the reversible case, the temperatures should be spaced so that
the acceptance rate of adjacent moves or swaps is approximately
0.234 just like for random-walk Metropolis
algorithms~\cite{RGG,statsci},
as derived in~\cite{atchade,sectemper}.
By contrast,
in the non-reversible case, the temperatures should be spaced so that
the acceptance rate of adjacent moves or swaps is approximately 0.387.

Now, the ratio of optimal $\ell$ values is $1.73/2.38 \doteq 0.73$,
corresponding to a 27\% decrease in proposal scaling standard deviation
for the non-reversible versus reversible case.
More importantly,
the ratio of optimal efficiency functions is $1.89 / 1.33 \doteq 1.42$,
corresponding to a 42\% increase in efficiency
for the non-reversible versus reversible case.
We record this formally as:

\begin{corollary}
\label{corfortytwo}
For the tempering algorithms as above,
the maximum efficiency for the non-reversible algorithm
is approximately 42\% more efficient than the reversible algorithm.
\end{corollary}

This corollary indicates that, when scaled with the optimal
choice of parameter $\ell$, the non-reversible case is indeed
more efficient than the reversible case, but not massively so.


We also have:

\begin{corollary}\label{infcor}
Under the above assumptions, as $\ell \searrow 0$
(corresponding to smaller and smaller temperature spacings),
the acceptance $\acc(\ell) \nearrow 1$,
and the ratio of efficiency
of non-reversible tempering
to reversible tempering converges to infinity.
\end{corollary}

\begin{proof}
By Theorem~\ref{curvethm},
the ratio of efficiency measures for non-reversible versus
reversible tempering is given by
$$
{
\acc(\ell) {4 \over c^2}
		\left[ \Phi^{-1}(\acc(\ell)/2)) \right]^2
\over
{\acc(\ell) \over 1-\acc(\ell)} {4 \over c^2}
		\left[ \Phi^{-1}(\acc(\ell)/2)) \right]^2
}
\ = \
{
1
\over
1 - \acc(\ell)
}
\ = \
{1 \over 1 - 2 \, \Phi(-c \, \ell / 2)}
\, .
$$
As $\ell \searrow 0$, we have
$\Phi(-c \, \ell / 2) \nearrow 1/2$, so
$\acc(\ell) \nearrow 1$, and
this efficiency ratio converges to $+\infty$, as claimed.
\end{proof}

Corollary~\ref{infcor} indicates that
the non-reversible algorithm becomes infinitely
more efficient that the reversible algorithm as the
proposal scaling becomes very small.
This observation corresponds to the result of \cite[Theorem~3]{syed} that,
as the mesh size goes to zero and number of temperatures $N\to\infty$, the
non-reversible algorithm achieves a higher-order roundtrip rate of $O(1)$,
compared to the reversible algorithm rate of $O(1/N)$.
However, when compared at their optimally scaled values,
the 42\% speed-up of Corollary~\ref{corfortytwo} gives a more
accurate measure of the relative improvement of using a
non-reversible tempering algorithm.

\section{Simulations}
\label{sec-sim}

To test our theory, we performed a detailed computer simulation of both
reversible and non-reversible tempering algorithms on a fixed
target in $d=100$ dimensions.
We performed a total of
$2 \times 10^{10}$
tempering iterations on
each of 20 different temperature spacing choices over the same
temperature range,
computed in parallel on the
Digital Research Alliance of Canada (DRAC) high-speed compute servers.
To conform to the above framework, we conducted the
simulation on a single target mode,
and counted the total number of round-trips
of the temperature from coldest to hottest and back again.
We used this count
to compute a rate of round-trips per million iterations.  The results
are shown in Figure~3.

\bigskip
\bigskip
\graphit{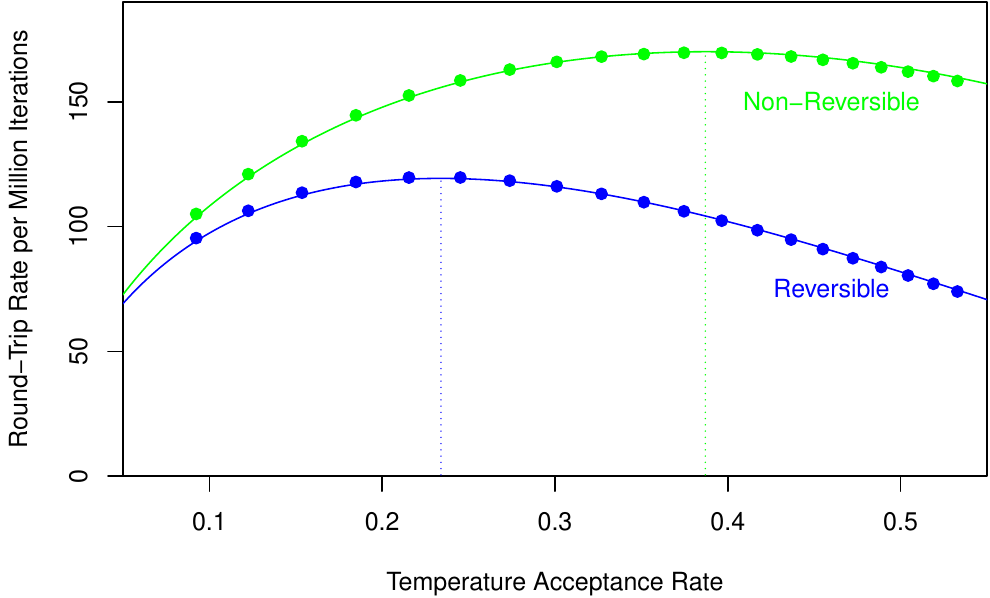}
\vskip -0.8cm
$\big.$
\begin{quote}
\baselineskip=12pt
{\bf Figure~3:}\
Simulated round-trip rates for non-reversible (top, green)
and reversible (bottom, blue) tempering algorithms as
a function of the temperature acceptance rate,
together with the theoretical relative efficiency curves
from Theorem~\ref{curvethm}
(scaled by a single appropriate constant multiplier),
and their
optimal scaling values (dashed lines) from Theorem~\ref{optthm},
showing excellent agreement.
\end{quote}
\bigskip

It is seen from Figure~3 that the simulated round-trip rates
show excellent agreement with the theoretical relative
efficiency curves from Theorem~\ref{curvethm}
(when scaled by a single appropriate constant multiplier, to convert the
relative efficiency measure into the round-trip rate).
This provides convincing evidence that our theoretical results
about relative efficiency of different tempering MCMC algorithms,
do indeed provide useful information about the practical information of
these algorithms to achieve round-trips between the coldest
and hottest temperatures.

\section{Proof of Theorem~\ref{brownianthm}}
\label{sec-proof}

Finally, we now prove Theorem~\ref{brownianthm}.
For notation, let $\Geom(C)$ be the probability distribution with
$\P(k)=(1-C)^k C$ for $k=0,1,2,3,\ldots$,
with expectation $(1-C)/C$ and variance $(1-C)/C^2$.
And, let $\mu$ be the probability distribution on $\pm 1$ with
$\mu(1) = A/(1-C)$ and $\mu(-1) = B/(1-C)$.
(In particular, if $B = 0$, then $\mu(1)=1$.)

\begin{lemma}
\label{difflemma}
Let $\{(X_n,Y_n)\}_{n=0}^\infty$ follow the Markov chain of Figure~1,
with initial vertical value $Y_0=+$.
Then there are i.i.d.\ random variables $G_n,H_n \sim \Geom(C)$,
and independent $\pm 1$-valued random variables $E_n,F_n \sim \mu$,
such that for all $n \ge 1$,
$$
X_{T_n} - X_{T_{n-1}} \ = \ \sum_{i=1}^{G_n} E_i - \sum_{i=1}^{H_n} F_i
\, ,
$$
where $T_0=0$ and $T_n = \sum_{i=1}^n (G_i+H_i+2)$.
\end{lemma}

\begin{proof}
Let $G_1$ be the last time just
\un{before} we first move to $\IZ \times\{-\}$.
Then $G_1 \sim \Geom(C)$.



Next, let $E_n$ be the increment in the $x$-direction from time $n-1$ to $n$,
\un{conditional} on staying in $\IZ \times\{+\}$,
so $E_n \sim \mu$.
Then the total displacement
\un{before} hitting $\IZ \times \{-\}$ is equal to
$\sum_{n=1}^{G_1} E_n$.
And, this takes $G_1$ steps, plus 1 step to move to $\IZ \times \{-\}$.

Similarly, the total displacement
\un{after} hitting $\IZ \times \{-\}$ but
\un{before} returning to $\IZ \times \{+\}$
is then equal to
$-\sum_{n=1}^{H_1} F_n$,
for corresponding time $H_1 \sim \Geom(C)$ and
independent increments $F_n \sim \mu$.
And this takes $H_1$ steps, plus 1 step to move to $\IZ \times \{-\}$.

It follows that $X_{T_1} := X_{G_1+H_1+2} = 
\sum_{i=1}^{G_1} E_i - \sum_{i=1}^{H_1} F_i$.
Continuing in this way, counting the times when the chain moves from $\IZ
\times \{+\}$ to $\IZ \times \{-\}$ and back, the result follows.
\end{proof}

\begin{lemma}
\label{diffvar}
With $T_n$ as in Lemma~\ref{difflemma}, for all $n \ge 1$ we have
$$
(i) \qquad
\E[T_n - T_{n-1}] \ = \ 2/C
\, ,
$$
and
$$
(ii) \qquad
\E[X_{T_n} - X_{T_{n-1}}] \ = \ 0
\, ,
$$
and
$$
(iii) \qquad
\Var[X_{T_n} - X_{T_{n-1}}] \ = \
2[(A-B)^2/C^2] + 2[(1-C)/C]
\, .
$$
\end{lemma}

\begin{proof}
$(i)$ \ With $G_n$ and $H_n$ as in 
Lemma~\ref{difflemma}, we compute that
$$
\E[T_n - T_{n-1}]
\ = \ \E[G_n + H_n + 2]
\ = \ \E[G_n] + \E[H_n] + 2
$$
$$
\ = \ [(1/C)-1] + [(1/C)-1] + 2
\ = \ 2/C
\, .
$$

$(ii)$ \ The quantities
$\sum_{i=1}^{G_n} E_i$ and $\sum_{i=1}^{H_n} F_i$ have the same
distribution, and hence the same mean, say $m$.  Therefore,
$$
\E[X_{T_n} - X_{T_{n-1}}]
\ = \ \E\Big[\sum_{i=1}^{G_n} E_i - \sum_{i=1}^{H_n} F_i\Big]
\ = \ m - m
\ = \ 0
\, .
$$

$(iii)$ \ Let $E_n$ and $F_n$ be as in Lemma~\ref{difflemma},
and let $S = \sum_{n=1}^{G_1} E_n$ be the total distance traveled
before first hitting $\IZ \times \{-\}$.

Then
$\E(E_n) = (A-B)/(1-C)$, and
$\E(E_n^2) = 1$, so
$\Var(E_n) = 1 - [(A-B)^2/(1-C)^2]$.



Hence, using the formula for variance of a random-sized sum from
Wald's identities,
$$
\Var(S)
\ = \ \E(G) \, \Var(E_1) + \Var(G) \, \E(E_1)^2
$$
$$
\ = \ [(1-C)/C] \, [1 - (A-B)^2/(1-C)^2] + [(1-C)/C^2] \, [(A-B)/(1-C)]^2
$$
$$
\ = \
[A^2 -2AB + B^2 + C - C^2] \, / \, C^2
\ = \
[(A-B)^2/C^2] + [(1-C)/C]
$$
\begin{equation}\label{varpiece}
\ = \
[(A-B)^2/C^2] + (A+B)/C
\, .
\end{equation}

Now, \eqref{varpiece} is the variance of a single piece,
i.e.\ the part before moving to $\IZ \times \{-\}$.
Then $X_{T_n}-X_{T_{n-1}}$ is formed by
combining two such pieces, of opposite sign.
Hence, its variance is twice the value in \eqref{varpiece}, as claimed.
\end{proof}


Putting these lemmas together, we obtain our diffusion result:



\begin{proof}[Proof of Theorem~\ref{brownianthm}]
In the language of
\cite{janson}, following \cite{serfozo},
the Markov chain described by Figure~1
has ``regenerative increments'' over the times $\{T_n\}$
specified in Lemma~\ref{difflemma},
with finite increment means and variances.
Hence, for fixed $t>0$,
the process $W(M) := X_{\lfloor Mt \rfloor}$ has regenerative increments at times $\{T_n/t\}$.
Then, it follows from~\cite[Theorem~1.4]{janson}
that as $M\to\infty$ with $t>0$ fixed, we have $W(M)/\sqrt{M} \equiv X_{\lfloor Mt \rfloor}/\sqrt{M} \to N(0,v)$,
where the corresponding volatility parameter $v$ is computed
(using Lemma~\ref{diffvar}) to be:
$$
v
\ = \ { \Var[X_{T_n}-X_{T_{n-1}}] \over \E[T_n-T_{n-1}] }
\ = \ { 2 [(A-B)^2/C^2] + 2 [(A+B)/C] \over 2/C }
\ = \ [(A-B)^2/C] + (A+B)
\, .
$$
This proves the claim about convergence to $N(0,vt)$ for fixed $t>0$.
(Strictly speaking, Lemmas~\ref{difflemma}
and Lemma~\ref{diffvar} assumed that the chain begins
in a state with $Y_0=+$, but
clearly the initial $Y_0$ value will not matter in the
$M\to\infty$ limit.)
The extended
claim about convergence of the entire process to Brownian motion then
follows from e.g.\
looking at just the second component in
\cite[Theorem~5]{gutjanson}.
This completes the proof.
\end{proof}


%



%
%
%
%
%
%

\bigskip
\noindent \bf Acknowledgements. \rm
We thank
Nick Tawn, Fernando Zepeda, Hugo Queniat, Saifuddin Syed,
Alexandre Bouchard-C\^ot\'e, Trevor Campbell,
Jeffrey Negrea, and Nikola Surjanovic for helpful discussions
about tempering issues, the latter four at the Statistical Society of Canada
annual conference in Newfoundland in June 2024. 
We thank Svante Janson for very useful guidance about invariance principles,
and Nick Tawn and Fernando Zepeda for
helpful conversations about non-reversible algorithms,
and David Ledvinka for useful suggestions,
and Duncan Murdoch for help with an R question.
We thank Daniel Gruner and Ramses van Zon of the
Digital Research Alliance of Canada
(DRAC) for technical assistance with parallel high-speed computing.
We acknowledge financial support from UKRI grant EP/Y014650/1
as part of the ERC Synergy project OCEAN,
by EPSRC grants Bayes for Health (R018561), CoSInES (R034710),
PINCODE (EP/X028119/1), and EP/V009478/1,
and from NSERC of Canada discovery grant RGPIN-2019-04142.



\begin{thebibliography}{99}

\bibitem{atchade}
Y.~Atchad\'e, G.O.~Roberts, and J.S.~Rosenthal (2011),
Towards Optimal Scaling of Metropolis-Coupled Markov Chain Monte Carlo.
Stat.\ and Comput.\ {\bf 21(4)}, 555--568.

\bibitem{nonrevMH}
J.~Bierkens (2016),
Non-reversible Metropolis-Hastings.
Stat.\ Comput.\ {\bf 26}, 1213--1228.

\bibitem{jorisgareth}
J.~Bierkens and G.O.~Roberts (2017),
A piecewise deterministic scaling limit of lifted Metropolis-Hastings
in the Curie-Weiss model.
Ann.\ Appl.\ Prob.\ {\bf 27(2)}, 846--882.

\bibitem{zigzag}
J.~Bierkens, P.~Fearnhead, and G.O.~Roberts (2019),
The Zig-Zag process and super-efficient sampling for Bayesian analysis
of big data.
Ann.\ Stat.\ {\bf 47(3)}, 1288--1320.

\bibitem{ABC2023}
M.~Biron-Lattes, T.~Campbell, and A.~Bouchard-C\"ot\'e (2023),
Automatic Regenerative Simulation via Non-Reversible Simulated Tempering.
arXiv:2309.05578

\bibitem{bouncy}
A.~Bouchard-C\^ot\'e, S.J.~Vollmer, and A.~Doucet (2018),
The bouncy particle sampler: a nonreversible rejection-free
Markov chain Monte Carlo method.
J.\ Amer.\ Stat.\ Assoc.\ {\bf 113(522)}, 855--867.

\bibitem{handbook}
S.~Brooks, A.~Gelman, G.~Jones, and X.-L.~Meng, eds.\ (2011),
Handbook of Markov chain Monte Carlo.
Chapman \& Hall, New York.

\bibitem{DHN}
P.~Diaconis, S.~Holmes, and R.M.~Neal (2000),
Analysis of a non-reversible Markov chain sampler.
Ann. Appl. Prob. {\bf 10(3)}, 726--752.

\bibitem{mcmcmc}
C.J. Geyer (1991), Markov chain Monte Carlo maximum likelihood.
In {\it Computing Science and Statistics: Proceedings of the 23rd
Symposium on the Interface}, 156--163.

\bibitem{mirageyer}
C.J.~Geyer and A.~Mira (2000),
On non-reversible Markov chains.  In N.~Madras (ed.),
Fields Institute Communications, Volume {\bf 26}: Monte Carlo Methods,
pp.~93--108.  Providence, RI: American Mathematical Society.

\bibitem{gutjanson}
A.~Gut and S.~Janson (1983),
The Limiting Behaviour of Certain Stopped Sums and Some Applications.
Scand.\ J.\ Stat.\ {\bf 10(4)}, 281--292.

\bibitem{hastings}
W.K.~Hastings (1970), Monte Carlo sampling methods using Markov chains
and their applications.  Biometrika {\bf 57}, 97--109.

\bibitem{janson}
S.~Janson (2023),
On a central limit theorem in renewal theory.
Preprint.  arXiv:2305.13229

\bibitem{marinari}
E.~Marinari and G.~Parisi (1992),
Simulated tempering: a new Monte Carlo scheme.
Europhys. Lett. {\bf 19}, 451--458.

\bibitem{metropolis}
N.~Metropolis, A.~Rosenbluth, M.~Rosenbluth, A.~Teller, and E.~Teller
(1953), Equations of state calculations by fast computing machines.  J.
Chem. Phys. {\bf 21}, 1087--1091.

\bibitem{suppress}
R.M.~Neal (1998), Suppressing Random Walks in Markov Chain Monte Carlo
Using Ordered Overrelaxation. In: Jordan, M.I. (eds), Learning in Graphical
Models, NATO ASI Series {\bf 89}. Springer, Dordrecht.

\bibitem{radfordbetter}
R.M.~Neal (2004), Improving Asymptotic Variance of MCMC Estimators:
Non-reversible Chains are Better.
Tech.\ Rep.\ {\bf 0406}, Dept.\ Statistics, University of Toronto.

\bibitem{simtemp}
E.~Marinari and G.~Parisi (1992),
Simulated tempering: a new Monte Carlo scheme.
Europhys. Lett. {\bf 19}, 451--458.

\bibitem{okabe}
T.~Okabe, M.~Kawata, Y.~Okamoto, and M.~Mikami (2001).
Replica-exchange Monte Carlo method for the isobaric-isothermal ensemble.
Chemical Physics Letters {\bf 335(5--6)}, 435--439.

\bibitem{RGG}
G.O. Roberts, A. Gelman, and W.R. Gilks (1997), Weak convergence and
optimal scaling of random walk Metropolis algorithms.
Ann. Appl. Prob. {\bf 7}, 110--120.

\bibitem{statsci}
G.O.~Roberts and J.S.~Rosenthal (2001),
Optimal scaling for various Metropolis-Hastings algorithms.
Stat.\ Sci.\ {\bf 16}, 351--367.

\bibitem{sectemper}
G.O.~Roberts and J.S.~Rosenthal (2014),
Minimising MCMC Variance via Diffusion Limits,
with an Application to Simulated Tempering.
Ann.\ Appl.\ Prob.\ {\bf 24(1)}, 131--149.
 
\bibitem{serfozo}
R.~Serfozo (2000),
Basics of Applied Stochastic Processes.
Springer-Verlag, Berlin.

\bibitem{sunetal}
Y.~Sun, F.~Gomez, and J.~Schmidhuber (2010),
Improving the asymptotic performance of Markov Chain Monte-Carlo
by inserting vortices.
Adv.\ Neur.\ Inform.\ Proc.\ Syst.\
{\bf 23}, 2235--2243.

\bibitem{pigeons}
N.~Surjanovic, M.~Biron-Lattes, P.~Tiede, S.~Syed,
T.~Campbell, A.~Bouchard-C\"ot\'e (2023),
Pigeons.jl: Distributed sampling from intractable distributions.
Preprint.  arXiv:2308.09769

\bibitem{ABC2022}
N.~Surjanovic, S.~Syed, A.~Bouchard-C\"ot\'e, 
and Trevor Campbell (2022),
Parallel Tempering With a Variational Reference.
Advances in Neural Information Processing Systems {\bf 35}.

\bibitem{ABC2024}
N.~Surjanovic, S.~Syed, A.~Bouchard-C\"ot\'e,
and Trevor Campbell (2024),
Uniform Ergodicity of Parallel Tempering with Efficient Local Exploration.
arXiv:2405.11384

\bibitem{swendsen}
R.H.~Swendsen and J.-S.~Wang (1987),
Nonuniversal critical dynamics in Monte Carlo simulations.
{\it Phys. Rev. Lett.} {\bf 58}, 86--88.

\bibitem{syed}
S.~Syed, A.~Bouchard-C\"ot\'e, G.~Deligiannidis, and A.~Doucet (2022),
Non-reversible parallel tempering: A scalable highly parallel MCMC scheme.
J. Roy. Stat. Soc. Ser. B {\bf 84(2)}, 321--350.

\bibitem{turitsyn}
K.S.~Turitsyn, M.~Chertkov, and M.~Vucelja (2011).
Irreversible Monte Carlo algorithms for efficient sampling.
Physica D:  Nonlinear Phenomena {\bf 240}, 410--414.


\end{thebibliography}
\end{document}